\documentclass[11pt,reqno,a4paper]{amsart}
\usepackage[english]{babel}
\usepackage[applemac]{inputenc}
\usepackage[T1]{fontenc}
\usepackage{palatino}
\usepackage{amsfonts}
\usepackage{amsmath}
\usepackage{amssymb}
\usepackage{amsthm}
\usepackage{graphicx}
\usepackage{wrapfig}
\usepackage{type1cm}
\usepackage[bf]{caption}
\usepackage{esint}
\usepackage{version}
\usepackage[active]{srcltx}  
\usepackage[usenames]{color}
\usepackage{caption}
\usepackage{tikz}
\usepackage{wasysym}
\usepackage{verbatim}
\usepackage{psfrag}

\newcommand{\R}{\mathbb{R}}
\newcommand{\Rd}{\mathbb{R}^{d\times d}}

\newcommand{\N}{\mathbb{N}}

\newcommand{\cM}{\mathcal{M}}
\newcommand{\cH}{\mathcal{H}}

\newcommand{\cE}{\mathcal{E}}
\newcommand{\cF}{\mathcal{F}}

\newcommand{\cG}{\mathcal{G}}
\newcommand{\cC}{\mathcal{C}}

\newcommand{\cU}{\mathcal{U}}
\newcommand{\cA}{\mathcal{A}}

\newcommand{\eps}{\varepsilon}

\renewcommand{\epsilon}{\varepsilon}
\newcommand{\e}{\varepsilon}
\renewcommand{\rho}{\varrho}
\renewcommand{\phi}{\varphi}

\newcommand{\w}{\wedge}

\newcommand{\p}{\mathtt{p}}

\newcommand{\bA}{\mathbf{A}}
\newcommand{\bB}{\mathbf{B}}
\newcommand{\qtext}[1]{\quad\text{#1}}
\newcommand{\interior}{\textrm{interior}}

\newcommand{\emdef}{\emph}

\theoremstyle{plain}
\newtheorem{thm}{Theorem}
\newtheorem{theorem}[thm]{Theorem}
\newtheorem{lemma}[thm]{Lemma}
\newtheorem{proposition}[thm]{Proposition}

\newtheorem{corollary}[thm]{Corollary}

\newtheorem*{claim*}{Claim}

\theoremstyle{definition}
\newtheorem{definition}{Definition}
\newtheorem{example}{Example}

\newtheorem*{examples*}{Examples}
\newtheorem*{example*}{Example}
\newtheorem{remark}{Remark}

\newtheorem*{notations*}{Notations}
\newtheorem*{notation*}{Notation}

\numberwithin{equation}{section}
\numberwithin{thm}{section}
\numberwithin{claim}{section}
\numberwithin{definition}{section}
\numberwithin{conjecture}{section}
\numberwithin{example}{section}
\numberwithin{remark}{section}
\numberwithin{notations}{section}
\numberwithin{notation}{section}

\author{De-Jun Feng}
\address{
Department of Mathematics,
The Chinese University of Hong Kong,
Shatin,  Hong Kong}

\email{djfeng@math.cuhk.edu.hk}

\author{Pablo Shmerkin}
\address{Torcuato Di Tella University, Buenos Aires, Argentina}
\email{pshmerkin@utdt.edu}

\thanks{The research of D.J. Feng  was partially supported by the RGC grant in CUHK; the research of P. Shmerkin  was partially supported by a Leverhulme Early Career Fellowship and by grant PICT 2011-0436 (ANPCyT)}
\title[Continuity of pressure for matrix cocycles]{Non-conformal repellers and the continuity of pressure for matrix cocycles}

\begin{document}

\begin{abstract}
The pressure function $P(\bA, s)$ plays a fundamental role in the calculation of the dimension of ``typical'' self-affine sets, where $\bA=(A_1,\ldots, A_k)$ is the family of linear mappings in the corresponding generating iterated function system. We prove that this function
depends continuously on $\bA$. As a consequence, we show that the dimension of  ``typical'' self-affine sets is a continuous function of the defining maps. This resolves a folklore open problem in the community of fractal geometry. Furthermore we extend the continuity result to more general sub-additive pressure functions generated by the norm  of matrix products or generalized singular value functions for matrix cocycles, and obtain applications on the continuity of equilibrium measures and the Lyapunov spectrum of matrix cocycles.
\end{abstract}

\maketitle

\section{Introduction and statement of main results}

\label{sec:introduction}

\subsection{Continuity of pressure for self-affine systems}

The topological pressure plays a key role in the dimension theory of dynamical systems. Indeed, there is a powerful heuristic principle, going back to Bowen and Ruelle \cite{Bowen79, Ruelle82}, that says that the dimension of a set $X$ invariant under a conformal map $f$ can be calculated as the (often unique) number $s$ satisfying $P(-s \log \|Df\|)=0$, where $Df$ denotes an appropriate notion of derivative for $f$. Although this formula has been shown to work in a very wide variety of settings (see, e.g., \cite{Barreira96}), the assumption that $f$ is conformal (in some, possibly weak, sense) cannot be dispensed with (in the case in which $f$ is hyperbolic, $f$ needs to be conformal on unstable leaves). We note that under quite general assumptions, in the conformal world the topological pressure $P(g)$ is a continuous functional of the function $g$ (in the appropriate topology), and thus the Hausdorff dimension of invariant sets varies continuously with the dynamics.

The dimension theory of non-conformal dynamical systems does not admit such a general principle. Nevertheless, different notions of pressure have been defined, which often give an upper bound for the dimension of the set, and sometimes give the right dimension in ``typical'' situations. However, as soon as one leaves the conformal situation, the norm of the derivative does not give sufficient information anymore, and one needs to consider finer geometric information given by the singular values of the Jacobian. Recall that if $A\in\R^{d\times d}$, the singular values $\alpha_1(A)\ge \cdots\ge \alpha_d(A)$ are the square roots of the eigenvalues of $A^* A$. Alternatively, they are the lengths of the semi-axes of the ellipsoid $A(B(0,1))$, where $B(0,1)$ is the unit ball in $\R^d$.

Before defining a general notion of non-conformal pressure, we start by considering an important special case that motivated this work. Fix an ambient dimension $d\ge 1$, an integer $k\ge 2$, and set
\begin{align*}
\cA_{d,k}&=\{ {\bf A}=(A_1,\ldots,A_k):A_i\in \R^{d\times d} \},\\
\cA^C_{d,k}&=\{ {\bf A}=(A_1,\ldots,A_k):A_i\in \R^{d\times d}, \|A_i\|<1\},\\
\cG_{d,k}&=\{ {\bf A}=(A_1,\ldots,A_k):A_i\in GL_d(\R) \},\\
 \cG^C_{d,k}&=\{ {\bf A}=(A_1,\ldots,A_k):A_i\in GL_d(\R), \|A_i\|<1\},
\end{align*}
where $GL_d(\R)$ denotes the collection of all  $d\times d$ invertible real matrices,  $\|\cdot\|$ denotes the standard Euclidean norm (the superscript $C$ stands for ``contraction'').
Given $\bA\in\cA^C_{d,k}$ and translations $t=(t_1,\ldots, t_k)\in \R^{kd}$, it is well-known that there exists a unique nonempty compact set $F=F(\bA,t)$ such that
\[
F = \bigcup_{i=1}^k A_i(F)+t_i.
\]
Such set is called a {\em self-affine set}. We note that if the pieces $A_i(F)+t_i$ do not intersect, then $F$ can be seen as a repeller for a piecewise affine expanding dynamical system, but the definition makes sense even if the pieces do intersect substantially.

Finding an exact general formula for the Hausdorff or box counting dimension of $F$ is considered to be an untractable problem. However, Falconer \cite{Falconer88} found an appropriate topological pressure equation for which the zero is always an upper bound for the dimension, and in many cases it is equal to the dimension. To state his result, for $0\le s\le d$, we define the {\em singular value function} $\phi^s:\R^{d\times d}\to [0,\infty)$ as
\[
\phi^s(A) = \alpha_1(A)\cdots\alpha_m(A)\alpha_{m+1}(A)^{s-m},
\]
where $m=\lfloor s\rfloor$. Here we make the convention $0^0=1$.  For completeness, if $s>d$, then we also define
\[
\phi^s(A) = |\det(A)|^{s/d}.
\]
The singular value function is sub-multiplicative, i.e. $\phi^s(AB)\le \phi^s(A)\phi^s(B)$; see \cite[Lemma 2.1]{Falconer88}. For $\bA\in\cA_{d,k}$, define
\[
P(\bA,s) = \lim_{n\to\infty} \frac{1}{n}\log\left( \sum_{|\mathbf{i}|=n} \phi^s(A_{\mathbf{i}}) \right) =\inf_n\frac{1}{n}\log\left( \sum_{|\mathbf{i}|=n} \phi^s(A_{\mathbf{i}}) \right) \in [-\infty,\infty),
\]
where $A_{\mathbf{i}}:=A_{i_1}\cdots A_{i_n}$ for $\mathbf{i}=i_1\cdots i_n$. (The existence of the limit and the second equality follow from the sub-multiplicativity of $\phi^s$.) Furthermore for $\bA\in\cA_{d,k}^C$, define
$$
s(\bA)=\inf\{s\geq 0: P(\bA,s)\leq 0\}.
$$
We call $s(\bA)$ the {\it singularity dimension} of $F(\bA,t)$. If $\bA\in \cG_{d,k}^C$, then $s(\bA)$ is the unique positive number $s$ such that $P(\bA,s)= 0$ (cf. \cite{Falconer88}).

For a set $F\subset \R^d$, let $\overline{\dim}_B (F)$, $\dim_B(F)$ and $\dim_H(F)$ denote the upper box counting dimension, the box counting dimension
 and the Hausdorff dimension of $F$, respectively (cf. \cite{Falconer03}).
We can now state Falconer's Theorem.
\begin{theorem}[{\cite[Theorem 5.3]{Falconer88}}]
\label{Falconer-thm}
Let $\bA\in\cA_{d,k}^C$.  Then  the following holds:

\begin{enumerate}
\item $\overline{\dim}_B(F(\bA,t))\le  \min(s(\bA),d)$ for all $t\in\R^{kd}$.
\item If $\|A_i\|<1/2$ for all $i$, then
\[
\dim_H(F(\bA,t)) = \dim_B(F(\bA,t))  = \min(s(\bA),d)
\]
for Lebesgue almost all $t\in\R^{kd}$.
\end{enumerate}
\end{theorem}

We remark that Falconer proved the second part under the assumption $\|A_i\|<1/3$ and $A_i\in GL_d(\R)$. Solomyak \cite{Sol98} later pointed out a modification in the proof that allows us to replace $1/3$ by $1/2$. By an observation of Edgar \cite{Edg92}, $1/2$ is optimal. The invertibility of the $A_i$ is not needed and Falconer's proof goes through without this assumption. Although Falconer's theorem holds only for generic translations, in several later works it was shown that the singularity dimension equals the box counting or the Hausdorff dimension for several concrete classes of self-affine sets, see \cite{Falconer92, HueterLalley95, KaenmakiShmerkin09}. This further highlights the significance of the singularity dimension.

In light of Theorem \ref{Falconer-thm}, the question of whether $s(\bA)$ and $P(\bA,s)$ depend continuously on $\bA$ arises naturally. Upper semi-continuity of $P(\cdot, s)$ and $s(\cdot)$ follows immediately from the definition (as $P(\cdot,s)$ is an infimum of continuous functions). Falconer and Sloan \cite{FalconerSloan09} proved the continuity of $P(\bA,s)$ in a number of special cases; in particular, they proved it under a rather strong irreducibility assumption on $\bA$ and, in the special case $s=1$, for upper triangular systems. Falconer and Sloan raised explicitly the question of continuity in general (though the problem has been known to the fractal geometry community long before that). In this article we resolve this problem.

\begin{theorem} \label{thm:continuity-pressure}
\begin{itemize}
\item[(1)]
For any $s\geq 0$, the map $\bA\to P(\bA,s)$ is continuous on $\cA_{d,k}$.
\item[(2)]The map $\bA\to s(\bA)$ is continuous on $\cA_{d,k}^C$.
\item[(3)]Moreover, the map $(\bA, s)\to P(\bA,s)$ is continuous at each point $(\bA,s)\in \cA_{d,k}\times [0,\infty)$ with $\bA\in \cG_{d,k}$ or $s\not\in \{0, 1,\ldots, d-1\}$.
\end{itemize}
\end{theorem}

\begin{remark}
The map $P(\cdot, \cdot)$ has some discontinuity points $(\bA, s)$ with $\bA\not\in \cG_{d,k}$ and $s\in \{1,\ldots, d-1\}$. For instance let $s\in \{1,\ldots, d-1\}$ and $\bA=(A, A,\ldots, A)$ for some non-invertible $d\times d$ matrix $A$  satisfying $\alpha_s(A)>0$ and $\alpha_{s+1}(A)=0$.  Then $(\bA,s)$ is a discontinuity point of $P$. Indeed, $P(\bA,s)\in \R$ and $P(\bA,t)=-\infty$ for any $t>s$.
\end{remark}

\subsection{Continuity of pressure for the norm of matrix products} Here we  present a related but in some sense simpler result concerning the norms of matrix products. Again let $\bA\in \cA_{d,k}$. Given $s>0$, define
\begin{equation*} \label{eq:def-matrix-pressure}
M(\bA,s) = \lim_{n\to\infty}\frac1n \log\left( \sum_{|\mathbf{i}|=n} \|A_{\mathbf{i}}\|^s \right).
\end{equation*}
The limit is easily seen to exist (and possibly equal $-\infty$) from sub-additivity.

The quantity $M(\bA,s)$ is rather natural. On one hand, it is closely linked to the Lyapunov exponent of an IID random matrix product (see, e.g. \cite[Chapter V]{BougerolLacroix85}), and to the Lyapunov spectrum of matrix products (cf. \cite{FengLau02, FengHuang10}). On the other hand, the limit as $s\to 0$  (which in the thermodynamic formalism is known as a ``zero temperature limit'') is the joint spectral radius of the matrices $(A_1,\ldots,A_k)$, which is an important quantity in a wide variety of fields, see for example \cite{HMST11} and references therein. Although the joint spectral radius is well-known to be continuous, it is far from clear from the definition whether $M(\bA,s)$ is always continuous. Nevertheless, similar to Theorem \ref{thm:continuity-pressure}, we have the following.

\begin{theorem} \label{thm:matrix-pressure-continuity}
$M(\bA,s)$  is  continuous on  $\cA_{d,k}\times [0,\infty)$.
\end{theorem}

The related problem of continuity of Lyapunov exponents of linear cocycles (as a function of the cocycle and of the matrix) is an ongoing problem of great interest, partly because of the connection to Schr\"{o}dinger operators. In a sense, the pressure is a more robust quantity as it is defined ``topologically'', without reference to a specific measure (although it can also be defined in terms of measures, see Section \ref{subsec:TF} below).

\subsection{A more general statement}
\label{subsec:generalpressure}

Falconer's result (Theorem \ref{Falconer-thm}) has been generalized in many directions; we will review some of these later in Section~\ref{sec:generalizations}.  Many of these more general results rely on an appropriate, correspondingly more general, notion of non-conformal topological pressure; this immediately raises whether continuity of the pressure can be established in a more general framework. We present our general result here, deferring a detailed discussion of its applications to Section~\ref{sec:generalizations}.

Let $(X,T)$ be a subshift of finite type (cf. \cite{Bowen75}). Here and in the rest of the article, subshifts of finite type are assumed to be defined on a finite alphabet. Recall that any map $A:X\to \R^{d\times d}$ induces a matrix cocycle: for $x\in X$ and $n\in\N$,
\[
A(x,n) =  A(T^{n-1}x)\cdots A(x).
\]
We denote the collection of finite words allowed in $X$ by $X^*$, and the subset of $X^*$ of words of length $n$ by $X^*_n$. A matrix cocycle $A$ on $X$ is said to be \emdef{locally constant} if $A(x)$ only depends on the first coordinate of $x$. Locally constant cocycles are naturally identified with elements of $\cA_{d,k}$.

Let $C(X)$ denote the collection of real continuous functions on $X$.  Given $g\in C(X)$,  a  cocycle $A:X\to\R^{d\times d}$ and $s\ge 0$, we define
\[
P_g(A,s) = \lim_{n\to\infty} \frac{1}{n}\log\left( \sum_{\mathbf{i}\in X^*_n} \;\sup_{y\in [\mathbf{i}]} \exp(S_ng(y))\phi^s(A(y,n)) \right),
\]
where $S_ng(x):=\sum_{i=0}^{n-1}g(T^ix)$.
The limit can be easily seen to exist, and equal the infimum, by sub-multiplicativity of the expression between parenthesis. In the important special case  $g\equiv 0$, we write $P(A,s)$ instead of $P_g(A,s)$.

Let $\cM(X, d)$ denote the collection of all $d\times d$ matrix cocycles on $X$.   We can now state the main result of the paper.
\begin{theorem} \label{thm:general-result}
Let $X$ be a subshift of finite type and $g\in C(X)$. Then the following statements hold: \begin{itemize}
\item[(1)]For fixed  $s>0$, any  locally constant cocycle $A$ is a continuity point of the pressure map $B\to P_g(B,s)$ on $\cM(X, d)$, in the $L^\infty(X,\R^{d\times d})$ topology.
 \item[(2)] Let $A\in \cM(X, d)$ be  locally constant and $s\geq 0$. If  $s\not\in \{0, 1,\ldots, d-1\}$, or if
 $A$ takes values in $GL_d(\R)$, then $(A,s)$ is a continuity point of the  map $(B,t)\to P_g(B,t)$ on $\cM(X, d)\times [0,\infty)$, where on $\cM(X,d)$ we consider the $L^\infty(X,\R^{d\times d})$ topology.
 \end{itemize}
\end{theorem}

We emphasize that, whenever $A$ is locally constant,  $P(B,s)$ is close to $P(A,s)$  if $B$ is uniformly close to $A$, even if $B$ is not locally constant or even discontinuous. 

We remark that the quantities $P(\bA,s)$, $M(\bA,s)$ and $P_g(A,s)$ are particular cases of topological pressures for sub-additive potentials; see Section~\ref{sec:sub-proofs} for details. The proofs of our continuity results make heavy use of dynamical systems theory, and in particular the sub-additive thermodynamic formalism. Moreover, since the topological pressure is a key component of the thermodynamic formalism, our results have applications to the continuity of other objects of interest, such as equilibrium measures, their entropies, and the Lyapunov spectrum of matrix cocycles. These will be presented in Section~\ref{sec:generalizations}.

The paper is structured as follows.  Section~\ref{sec:linear-algebra} contains some  preliminaries on linear algebra and the cone condition, which will be needed in our proofs. In Section~\ref{sec:cone-condition}, we  construct a ``large'' subsystem of a matrix cocycle which satisfies the cone condition.  Section~\ref{sec:sub-proofs} contains background on the variational principle for the sub-additive topological pressure, and the proofs of our main results. Further generalizations, applications and remarks are given in Section~\ref{sec:generalizations}.

\section{Linear algebraic preliminaries and the cone condition}

\label{sec:linear-algebra}

\subsection{Cones and the cone condition}

The proofs of our main results will be based on finding a sub-system of (an iteration of) the original system which satisfies the cone condition. In this section we deal with elementary linear-algebraic facts related to this. Although most of the material is standard, proofs are given for the convenience of the reader, and since it is difficult to trace down the exact statements in the literature.

A \emdef{cone} $K$ in a finite-dimensional Banach space is a nonempty, convex, closed subset such that $t v\in K$ whenever $t>0, v\in K$, and $K\cap -K=\{0\}$.

\begin{definition}
Let $A:X\to\R^{d\times d}$ be a map generating a linear cocycle. We say that $A$ satisfies the \emdef{cone condition} with cones $K, K'$, if  $K'\setminus\{0\}\subset \interior(K)$, and $A(x)K\subset (K'\cup -K')$ for all $x\in X$.
\end{definition}

The importance of the cone condition for us comes from the following fact.
\begin{lemma} \label{lem:cone-implies-almost-multiplicative}
Let $V$ be a finite dimensional Banach space, and let $K,K'\subset V$ be cones with $K'\setminus\{0\}\subset \interior(K)$. Then there is a constant $c=c(K,K')>0$ such that for any pair of linear maps $A_1, A_2$ such that $A_i K\subset (K'\cup -K')$ $(i=1,2)$,
\[
c \|A_1\| \| A_2\| \le \|A_1 A_2\| \le \|A_1\|\|A_2\|,
\]
{ where $\|A\|$ denotes the operator norm of $A$}.
\end{lemma}
\begin{proof}
This is elementary and most likely known, but we have not been able to find a reference so a proof is given for the reader's convenience. We only need to prove the left-hand inequality.

We first claim that there exists $c_1>0$ such that for any linear map $A:V\to V$ such that $A(K)\subset (K'\cup -K')$ and any $w\in K'\setminus\{0\}$,
{
\[
\frac{|A w|}{|w|}\geq c_1\|A\|.
\]
}
Indeed, suppose this is not the case. Then, for all $n$ we can find a linear map $A_n$ of norm $1$ with $A_n(K)\subset K'$, and $w_n\in K'$ also of norm $1$, such that $\| A_n w_n\|< 1/n$. By compactness, this implies that there are a linear map $A$ on $V$ of norm $1$ (in particular nonzero) such that $A(K)\subset K'$, and a vector $w\in K'$ such that $Aw=0$. Now pick $u\in K$ such that $Au\neq 0$ and $w-u\in K$; this is possible since $K'\setminus\{0\}\subset\interior(K)$. It follows that { $Au=-A(w-u)\in -K'$}, whence $Au\in K'\cap -K'$, contradicting that $K'$ is a cone.

Now the lemma follows easily since, for a fixed $w\in K'$ of unit norm,
{
\[
\| A_1 A_2\| \ge  |A_1 A_2 w| \ge c_1 \| A_1\| |A_2w| \ge c_1^2 \|A_1\| \|A_2\|.
\]
}
\end{proof}

\subsection{Exterior algebra}

As usual in the study of matrix cocycles, we often make use of the exterior algebra generated by the $j$-alternating forms, which we denote $(\R^d)^{\wedge j}$. It is endowed with an inner product $(\cdot|\cdot)$, with the property that
\[
(v_1\wedge\cdots\wedge v_j|w_1\wedge\cdots\wedge w_j) = \det(v_a\cdot w_b)_{1\le a,b\le j},
\]
where $v_a\cdot w_b$ is the usual inner product on $\R^d$. In particular, if $\{ v_i\}_{i=1}^j$ are orthonormal, then $v_1\wedge\cdots\wedge v_j$ has norm $1$. For $A\in\R^{d\times d}$, we recall that the $j$-fold exterior product $A^{\wedge j}$ of $A$ is defined by the condition
\[
A^{\wedge j}(v_1\wedge\cdots\wedge v_j) = A v_1\wedge\cdots \wedge A v_j.
\]
The following properties are well known:
\begin{enumerate}
\item $(AB)^{\w j}=A^{\w j} B^{\w j}$, and in particular $\| (AB)^{\w j}\|\le \| A^{\w j}\| \| B^{\w j}\|$.
\item $\| A^{\w j}\|=\alpha_1(A)\cdots \alpha_j(A)$, and in particular $\| A^{\w j}\|\le \| A\|^j$.
\end{enumerate}

\subsection{Block-diagonal matrices}
It is geometrically clear that if $A\in\R^{d\times d}$ has an eigenvector $v$ of unit norm such that $|Av| \gg \|A\|_{v^\perp}$ (where $\|A\|_{v^\perp}$ denotes the norm of the restriction of $A$ to the hyperplane orthogonal to $v$), then $A$ satisfies the cone condition for suitable  conical neighborhoods of (the half-line containing) $v$. Because of the robustness of the cone condition, this continues to be true if $v$ is only an approximate eigenvector (meaning that the direction of $Av$ is close to that of $v$), and a similar approximation for $v^\perp$. In this section we prove a statement of this kind for certain exterior products in the ``exact'' case; the ``approximate'' case is deduced as a consequence in the next section. We begin with some definitions and basic lemmas.

A matrix $A\in\R^{d\times d}$ will be called \emdef{$(\lambda,\e)$-conformal} if
\[
\exp(\lambda-\e)|x| \le |Ax| \le \exp(\lambda+\e)|x|\qtext{for all }x\in\R^d.
\]
This is equivalent to saying that all singular values of $A$ lie between $\exp(\lambda-\e)$ and $\exp(\lambda+\e)$.

Given $H_i\in\R^{d_i\times d_i}$, $i=1,\ldots,p$, we define their direct sum $H=\bigoplus_{i=1}^p H_i$ as the block-diagonal matrix with blocks $H_1,\ldots, H_p$. Thus $H\in\R^{d\times d}$, where $d=\sum_{i=1}^p d_i$.

We say that $H$ is of \emdef{hyperbolic class $(d_1,\tau_1),\ldots, (d_p,\tau_p)$ with tolerance $\e$} if:
\begin{itemize}
\item $H=\bigoplus_{i=1}^p H_i$,
\item $H_i\in\R^{d_i\times d_i}$ is $(\tau_i,\e)$-conformal,
\item $\tau_i-\e> \tau_{i+1}+\e$ for $i=1,\ldots,p-1$.
\end{itemize}
The family of all such $H$ will be denoted by $\cH_{(d_i,\tau_i)}^\e$.

\begin{lemma} \label{lem:sing-values-hyperbolic}
Suppose  $H\in \cH_{(d_i,\tau_i)}^\e$.  Let $d_0=0$.  If  $\sum_{i=0}^{r-1} d_i< j \le \sum_{i=0}^{r} d_i$ for some $r\in \{1,\ldots, p\}$, then
\[
\exp(\tau_{r}-\e) \le \alpha_j(H) \le \exp(\tau_{r}+\e).
\]
\end{lemma}
\begin{proof}
If $H_i =  A_i D_i B_i$ is a singular value decomposition of $H_i$ (that is, $A_i, B_i$ are orthonormal and $D_i$ is diagonal), then
\begin{equation} \label{eq:SVD}
H = \left(\bigoplus_i A_i\right)\left(\bigoplus_i D_i\right)\left(\bigoplus_i B_i\right)
\end{equation}
 is a singular value decomposition of $H$. This reduces the statement to the case in which $H$ is diagonal, which is obvious.
\end{proof}

\begin{lemma} \label{lem:wedge-contractivity}
Let $H\in \cH_{(d_i,\tau_i)}^\e$. Let $t=\sum_{i=1}^r d_i$ for some $1\le r\le p$. Write $\Gamma=\sum_{i=1}^r d_i\tau_i$. Let $\{e_1, \ldots, e_d\}$ be  the canonical basis of $\R^d$.  Then $e_1\wedge\cdots\wedge e_t$ is an eigenvector of $H^{\w t}$ with eigenvalue $\ge  \exp\left(\Gamma - t\e\right)$ in modulus, and
\[
|H^{\w t}(e_{i_1}\wedge\cdots\wedge e_{i_t})| \le \exp\left(\Gamma +\tau_{r+1}-\tau_r + t\e\right),
\]
whenever $1\leq i_1<\cdots <i_t\leq d$ and $(i_1,\ldots, i_t)\neq (1,\ldots,t)$.
\end{lemma}
\begin{proof}
Recall that, since $H_i\in \R^{d_i\times d_i}$, the $d_i$ exterior product of $H_i$ is simply multiplication by the determinant. Identifying $H_i$ with the restriction of $H$ to the corresponding subspace of $\R^d$, we deduce that
\[
H^{\w t}(e_1\wedge\cdots\wedge e_t) = \left(\prod_{i=1}^r \det(H_i)\right) (e_1\wedge\cdots\wedge e_t),
\]
which implies the first claim.

For the second part, let $H= ADB$ be a singular value decomposition of $H$ with the form \eqref{eq:SVD}. It follows that
\[
|H^{\w t}(e_{i_1}\wedge\cdots\wedge e_{i_t})| = | D^{\w t} B^{\w t}(e_{i_1}\wedge\cdots\wedge e_{i_t})|.
\]
In this case, we have that
\[
B^{\w t}(e_{i_1}\wedge\cdots\wedge e_{i_t}) = \bigwedge_{j=1}^p w_1^j \wedge \cdots \wedge w_{t_j}^j,
\]
where, for each $j$, $\{ w_i^j\}$ is an orthonormal system in $\R^{d_j}$ (embedded in $\R^d$ as before), possibly empty; however, at least one of them is nonempty for some $j>r$ (here we use that $(i_1,\ldots, i_t)\neq (1,\ldots, t)$). Using the fact that each $w_1^j \wedge \cdots \wedge w_{t_j}^j$ is a unit vector in $(\R^{d_j})^{\w t_j}$, and the block structure once again, we estimate
\begin{align*}
|H^{\w t}(e_{i_1}\wedge\cdots\wedge e_{i_t})| &\le \prod_{j=1}^p \left\| D_j^{\w t_j}\right\|\\
 &\le \prod_{j=1}^p \|D_j\|^{t_j}\\
 &\le \prod_{j=1}^p \exp(t_j(\tau_j+\e))\\
 &= \exp\left(\left(\sum_{j=1}^p t_j\tau_j\right) +t\e\right)\\
 &\le \exp(\Gamma+\tau_{r+1}-\tau_r+t\e),
\end{align*}
where in the last step we used that $t_j>0$ for some $j>r$.
\end{proof}

In order to apply the last result, we recall some simple facts in an abstract setting. Let { $(V,\cdot|\cdot)$} be a finite dimensional inner product space, and denote the associated norm by $|\cdot|$. Given a nonzero vector $v\in V$ and $0<r<1$, we define the $r$-cone around $v$ as
\[
\cC(v,r) = \{ w\in V: (v|w) \ge (1-r)|v||w|  \}.
\]

\begin{lemma} \label{lem:cone-condition}
Let $A:V\to V$ be a linear map and let { $v\in V\backslash\{0\}$} and $\lambda>0$ be such that
\[
Av=\lambda v \qtext{ and }\quad \|A|_{v^\perp}\|<\lambda/{ 18}.
\]
Then ${ A}\cC(v,1/2)\subset \cC(v,1/5)$.
\end{lemma}
\begin{proof}
{
Without loss of generality,  assume that $|v|=1$.  Let $w\in \cC(v,1/2)$ and $w\neq 0$. We can write $w=av+bz$ for some $a,b\in \R$ and $z\in v^\perp$ of unit norm. The assumption $w\in \cC(v,1/2)$ implies that
\[
a=(v|w)\geq (1-1/2)|w|=\frac{1}{2}\sqrt{a^2+b^2},
\]
i.e. $\sqrt{3}a\geq |b|$.
 A direct calculation shows that
\[
\frac{(v|Aw)}{|v||Aw|}=\frac{(v|(a\lambda v+bAz))}{|a\lambda v+bAz|} \ge \frac{a\lambda -|b|\lambda/18}{a\lambda +|b|\lambda/18} \ge \frac{1-\sqrt{3}/18}{1+\sqrt{3}/18}> 4/5.
\]
Hence $Aw\in \cC(v,1/5)$.
}
\end{proof}

\begin{corollary} \label{cor:block-diagonal-satisfies-cone}
Let $H\in \cH_{(d_i,\tau_i)}^\e$, and suppose that
\[
\sqrt{d\choose {\lfloor d/2\rfloor }}\max_{1\le r\le p-1}\exp(-\tau_r+\tau_{r+1}+2 d\e) < 1/18.
\]
Denote $t_r=\sum_{i=1}^r d_i$ and $e^{\w t}=e_1\wedge\cdots\wedge e_t$. Then, for all $1\le r\le p-1$,
\[
H^{\w t_r}\left(\cC(e^{\w t_r},1/2)\right) \subset (\cC(e^{\w t_r},1/5)\cup -\cC(e^{\w t_r},1/5)).
\]
\end{corollary}
\begin{proof}
{ Fix $1\le r\le p-1$ and set $t=t_r$. Let $\cA$ denote the collection of indices $(i_1,\ldots,i_t)$ with $1\leq i_1<\cdots<i_t\leq d$. Then
$\#\cA={d\choose  t}\leq {d\choose {\lfloor d/2\rfloor}}$. Set $v=e^{\w t}$.
By Lemma \ref{lem:wedge-contractivity}, $v$ is an eigenvector of $H^{\w t}$ with eigenvalue $\geq  \exp(\Gamma-t\epsilon)$ in modulus, where $\Gamma=\sum_{i=1}^r d_i\tau_i$. Moreover
$$
|H^{\w t}(e_{i_1}\wedge\cdots\wedge e_{i_t})|
 \le \exp(\Gamma+\tau_{r+1}-\tau_r+t\e)
$$
for any $(i_1,\ldots,i_t)\in \cA\setminus \{(1,\ldots, t)\}$.
Observe that
$$
\{e_{i_1}\wedge\cdots\wedge e_{i_t}:\; (i_1,\ldots,i_t)\in \cA\}
$$
is an orthonormal basis of $\R^{\w t}$.  By the Cauchy-Schwartz inequality, we have
\begin{align*}
\|H|_{v^{\perp}}\|&\leq \sqrt{\#\cA}\sup\{|H(e_{i_1}\wedge\cdots\wedge e_{i_n})|:\; (i_1,\ldots,i_t)\in \cA\setminus \{(1,\ldots, t)\}\}\\
&\leq \sqrt{d\choose {\lfloor d/2\rfloor }} \exp(\Gamma+\tau_{r+1}-\tau_r+t\e)\leq   \exp(\Gamma-t\e)/18.
\end{align*}
Now the corollary follows from  Lemma \ref{lem:cone-condition}.}
\end{proof}

\subsection{Space of splittings}

The Grassmanian $G(d,k)$ is the manifold of $k$-dimensional subspaces of $\R^d$. It is endowed with the metric $d(V,W)=\inf \|O-I\|$, where the infimum is over all orthogonal maps $O$ with $OV=W$.

Given numbers $ d_1,\ldots,  d_p\in\N$ with $\sum_i  d_i=d$, we denote by $\mathcal{S}=\mathcal{S}_{ d_1,\ldots, d_p}$ the collection of all splittings $\mathbf{V}= \bigoplus_{i=1}^p V_i$ of $\R^d$, where $V_i\in G(d, d_i)$. Then $\mathcal{S}$ is naturally identified with an open subset of the product $\times_{i=1}^p G(d, d_i)$ and inherits the $\ell^\infty$ metric: $d(\mathbf{V},\mathbf{W}) = \max_{i=1}^p d(V_i,W_i)$. We fix the canonical splitting $\mathbf{V}^* = \bigoplus_{i=1}^p V_i^*$, where $V_i^*$ is the subspace generated by the canonical vectors $e_j$, { $ d_0+\ldots+ d_{i-1} < j\le d_0+ d_1+\ldots+ d_{i}$ with $d_0:=0$}.

\begin{lemma} \label{lem:cone-condition-general}
There is $\delta=\delta(d)>0$ such that the following holds. Let $A:\R^d\to \R^d$ be a linear map such that:
\begin{enumerate}
\item \label{it:splitting-near-canonical} There are splittings $\mathbf{V},\mathbf{W}$ such that $A\mathbf{V}=\mathbf{W}$ (i.e. $A V_i=W_i$) and
\[
d(\mathbf{V},\mathbf{V}^*),d(\mathbf{W},\mathbf{V}^*)<\delta.
\]
\item \label{it:contraction-on-splitting} $|Av|\in [\exp(\tau_i-\e)|v|,\exp(\tau_i+\e)|v|]$ for all $v\in V_i\setminus\{0\}$, where the $\tau_i$ and $\e$ satisfy
\[
\sqrt{d\choose {\lfloor d/2\rfloor }}\max_{1\le r\le p-1}\exp(-\tau_r+\tau_{r+1}+2 d\e) < 1/18.
\]
\end{enumerate}
Then
\[
{ A}^{\w t_r}\left(\cC(e^{\w t_r},1/3)\right) \subset (\cC(e^{\w t_r},1/4)\cup -\cC(e^{\w t_r},1/4))\quad\text{for } r=1,\ldots,p-1,
\]
where $t_r$ and $e^{\w t_r}$ are as in Corollary \ref{cor:block-diagonal-satisfies-cone}.
\end{lemma}
\begin{proof}
Let $O_i$ be an orthogonal map of $\R^{ d}$ such that $O_i V^*_i=V_i$ and $\|O_i-I\|\le \delta$, and let $O$ be linear map that equals $O_i$ on $V^*_i$. If $\delta$ is small, different subspaces $V_i$ make an angle close to $\pi/2$, and it follows that $\|O-I\|<C_d \delta$. Likewise, we define $O'$ such that $O'|_{V^*_i}$ is an orthogonal map onto $W_i$ and $\|O'-I\|<C_d \delta$.

Now, by \eqref{it:splitting-near-canonical} and \eqref{it:contraction-on-splitting}, $H:=(O')^{-1} A O\in\mathcal{H}_{\tau_i,d_i}^\e$ whence, by Corollary \ref{cor:block-diagonal-satisfies-cone},
\[
H^{\w t_r}\left(\cC(e^{\w t_r},1/2)\right) \subset (\cC(e^{\w t_r},1/5)\cup -\cC(e^{\w t_r},1/5))\quad\text{for } r=1,\ldots,p-1.
\]
Then we only need to pick $\delta$ small enough to ensure that if $\|U-I\|<C_d \delta$, then
\[
U\left(\cC(e^{\w t},{ 1/3})\right)  \subset \cC(e^{\w t}, { 1/2})\,\,\text{ and }\,\,U\left(\cC(e^{\w t},1/5)\right)\subset \cC(e^{\w t},1/4)\text{ for } t=1,\ldots,d.
\]
\end{proof}

\section{Sub-cocycles satisfying the cone condition}

\label{sec:cone-condition}

From now on $X$ is a fixed subshift of finite type, and the shift map is denoted by $T$. The family of $T$-invariant ergodic measures on $X$ will be denoted by $\cE$.  We will require the following version of Oseledets Ergodic Theorem, due to Froyland, Lloyd and Quas \cite[Theorem 4.1]{FLQ10}:

\begin{theorem} \label{thm:oseledets}
Given a measurable map $A:X\to\Rd$ and $\mu\in\cE$ such that
\[
\int \log^+\|A(x)\|d\mu(x)<+\infty,
\]
there exist  $\lambda_1>\cdots>\lambda_p\ge-\infty$, dimensions $d_1,\ldots, d_p$ with $\sum_{i=1}^p d_i=d$, and a measurable family of splittings $\mathbf{E}(x)=\bigoplus_{i=1}^p E_i(x)\in\mathcal{S}_{d_1,\ldots,d_p}$, such that for $\mu$-almost all $x$ the following holds:
\begin{enumerate}
\item $A(x)E_i(x) \subset E_i(Tx)$, with equality if $\lambda_i > -\infty$,
\item For all $v\in E_i(x)\setminus\{0\}$,
\[
\lim_{n\to\infty} \frac{\log |A(x,n)v|}{n}=\lambda_i,
\]
with uniform convergence on any compact subset of $E_i(x)\setminus\{0\}$.
\end{enumerate}
\end{theorem}
We remark that the uniform convergence in part (2) of Theorem \ref{thm:oseledets} is not stated in \cite{FLQ10}. However, this is well-known when the cocycle  $A$ takes values in $\mathrm{GL}_d(\R)$, and we sketch an argument that works also in the general case of $\R^{d\times d}$-valued cocycles. It follows from standard proofs of Oseledets' Theorem (see e.g. \cite[Theorem 1.6]{Ruelle79}) that the limit
\begin{equation} \label{eq:oseldets-convergence-to-matrix}
\lim_{n\to \infty} (A(x,n)^*A(x,n))^{1/2n}:=B
\end{equation}
exists for $\mu$ almost all $x$, and $\exp(\lambda_1)>\ldots> \exp(\lambda_p)$ are the different eigenvalues of $B$ (the matrix $B$ depends on $x$). Moreover, letting $U_1,\ldots, U_\p$ denote the corresponding eigenspaces of $B$,  we have  $(1/n)\log |A(x,n)v|$ converges to $\lambda_i$ uniformly (with respect to $v$) on any compact subset of $U_i\backslash \{0\}$, as $n\to \infty$.  According to Theorem \ref{thm:oseledets}(2), we have $E_i\subset U_i\oplus U_{i+1}\oplus \cdots \oplus U_p$ and $E_i\cap (U_{i+1}\oplus \cdots \oplus U_p)=\{0\}$; moreover, for any compact set $K_i$ of $E_i(x)\backslash \{0\}$, the orthogonal projection of $K_i$ on $U_i$ is a compact subset of $U_i\backslash \{0\}$, from which we deduce the uniform convergence of   $(1/n)\log |A(x,n)v|$  on $K_i$.

The data $(\lambda_i,d_i)_{i=1}^p$ from Theorem \ref{thm:oseledets} is called the \emdef{Lyapunov spectrum} of $(A,\mu)$. When $A(x)$ is invertible for all $x$ this is the classical Oseledets Ergodic Theorem, but we underline that the above is valid even in the non-invertible case (in which case the usual statements of Oseledets' Theorem only provide a flag and not a splitting into subspaces).

We record the following useful fact.
\begin{lemma} \label{lem:oseledets-gives-Lyapunov-for-SV}
Let $A:X\to \R^{d\times d}$ be a measurable cocycle such that $\int \log^+\|A(x)\|\,d\mu(x)<+\infty$, and let $\mu\in\cE$. Let $\{(\lambda_i,d_i)\}_{i=1}^p$ be the Lyapunov spectrum, and write { $d_0=0$} and  $t_r=\sum_{i={ 0}}^r d_i$, $\Gamma_r = \sum_{i={ 0}}^r d_i \lambda_i$. If $t_r < s\le t_{r+1}$ for some $0\le r<p$, then
\[
\lim_{n\to\infty}\frac{1}{n}\int \log\phi^s(A(x,n)) \,d\mu(x) = \Gamma_r + (s-t_r)\lambda_{r+1}.
\]
\end{lemma}
\begin{proof}
{
Fix $r\in \{0,1,\ldots, p-1\}$ and $s\in (t_r,  t_{r+1}]$. Let $m=\lfloor s \rfloor$.   For any $B\in \R^{d\times d}$, we have
\begin{align}
\phi^s(B)=&\alpha_1(B)\ldots \alpha_{m}(B)\alpha_{m+1}(B)^{s-m}\nonumber\\
=&(\alpha_1(B)\ldots \alpha_{m}(B))^{m+1-s}(\alpha_1(B)\ldots \alpha_{m+1}(B))^{s-m}\nonumber\\
=&\|B^{\w m}\|^{m+1-s}\|B^{\w (m+1)}\|^{s-m}, \label{e-inequality}
\end{align}
 using the fact that $\|B^{\w k}\|=\alpha_1(B)\ldots \alpha_{k}(B)$.
 By \eqref{eq:oseldets-convergence-to-matrix} and the subadditive ergodic theorem,
\begin{equation}
\label{eq:Ruelle}
\lim_{n\to\infty}\frac{1}{n}\int \log \| A(x,n)^{\w k}\|d\mu(x) = \Gamma_r+(k-t_r)\lambda_r, \quad k=m, m+1.
\end{equation}
Now the lemma follows from \eqref{e-inequality} and  \eqref{eq:Ruelle}.
}
\end{proof}

The following theorem contains the main idea in the proof of Theorem \ref{thm:general-result}: after suitable iteration, we may find a ``large'' subsystem which satisfies the cone condition for appropriate exterior powers. The idea of using recurrence in connection with Oseledets' Theorem was inspired by the proof of \cite[Theorem 15]{AvilaBochi02}. Invariant cones in connection with matrix cocycles and Oseledets' Theorem were also recently employed by Kalinin in \cite{Kalinin11}, although the setting there is different.

\begin{theorem} \label{thm:sub-cocycle-cone-condition}
 Fix a locally constant cocycle $A:X\to \R^{d\times d}$, and  an ergodic measure $\mu\in\cE$. Let $\lambda_i, 1\le i\le p$   be the distinct Lyapunov exponents in decreasing order, and write $d_i$ for the multiplicity of $\lambda_i$. Further, let $t_r=\sum_{i=1}^r d_i$.

  Then there exist $\eta>0$, a set $S\subset\N$ with bounded gaps, and families $\{\Sigma_n:n\in S\}$ of finite words such that:
 \begin{enumerate}
 \item \label{it:concatenations-allowed} $\Sigma_n\subset X^*_n$; moreover, concatenations of arbitrary words in $\Sigma_n$ are in $X^*$.
 \item \label{it:measure-bounded-below} $\sum_{{\mathbf i}\in\Sigma_n} \mu[{\mathbf i}]\ge \eta$.

 \item \label{it:cone-condition-holds} Moreover  if $p>1$ (i.e. there are at least two distinct Lyapunov exponents), then there exist cones $K_r, K'_r\subset \R^{\w t_r}$, $1\le r\le p-1$, with $K'_r\setminus\{0\}\subset {\interior}(K_r)$ such that $A^{\w t_r}(x,n)K_r\subset (K'_r\cup -K'_r)$ whenever $n\in S$ and $x|n\in \Sigma_n$.
 \end{enumerate}
\end{theorem}
\begin{proof}
After a change of coordinates (which does not affect the statement) we may assume that the canonical splitting $\mathbf{V}^*$ is in the support of the  push-down measure $\mathbf{E}_*\mu$ on the space of splittings. Let $\delta$ be the constant from Lemma \ref{lem:cone-condition-general}. Then $\mu(\mathbf{E}^{-1}(B(\mathbf{V}^*,\delta)))>0$, and we can pick a symbol $a$  in the alphabet of $X$  such that $\mu(Y)>0$, where $Y=\mathbf{E}^{-1}(B(\mathbf{V}^*,\delta))\cap [a]$.

Let $\e=\frac18\min_{r=1}^{p-1} (\lambda_{r+1}-\lambda_r)$. By Egorov's Theorem and Theorem \ref{thm:oseledets}, there are $N_0\in\N$ and a set $Z\subset X$ with
\[
\mu(Z)>1-\frac{1}{3}\mu(Y)^2,
\]
such that if $x\in Z$ and $n\ge N_0$, then
\[
|A(x,n)v| \in [\exp(n(\lambda_i-\eps))|v|,\exp(n(\lambda_i+\eps))|v|]\quad \text{for all }v\in E_i(x),
\]
for all $i$ such that $\lambda_i>-\infty$. By making $N_0$ larger if necessary, we may assume that $N_0\ge 2d$ and $\exp(N_0\eps) \ge \sqrt{d\choose {\lfloor d/2\rfloor }}$.

By Khintchine's recurrence theorem (see e.g. \cite[Theorem 3.3]{Petersen83}), the set
\[
S:=[N_0,\infty) \cap \{n: \mu(Y\cap T^{-n}Y)> \mu(Y)^2/2\}
\]
has bounded gaps. We let
\[
\Sigma_n=\{ x|n: x\in Y\cap T^{-n}Y\cap Z\}.
\]
By definition, each word in $\Sigma_n$ is allowed and has length $n$. Moreover, since all sequences in $Y$ start with $a$, if $x\in Y\cap T^{-n}Y$ then $x_1=a$ and the transition $x_n\to a$ is allowed in $X$. This shows that concatenations of words in $\Sigma_n$ are allowed. Thus \eqref{it:concatenations-allowed} holds. Also,
\[
\sum_{{\mathbf i}\in\Sigma_n} \mu[{\mathbf i}] \ge \mu(Y\cap T^{-n}Y\cap Z)\ge \frac{\mu(Y)^2}{6}=:\eta>0,
\]
which yields \eqref{it:measure-bounded-below}.

Next assume that $p>1$.
Taking stock, if $x\in Y\cap T^{-n}Y\cap Z$ and $n\ge N_0$, then $A(x,n)$ satisfies the assumptions of Lemma \ref{lem:cone-condition-general}, with $\mathbf{V}=\mathbf{E}(x)$, $\mathbf{W}=\mathbf{E}(T^n x)$ and $\tau_i=n\lambda_i$. Hence $A^{\w t_r}(x,n)K_r\subset (K'_r\cup  -K'_r)$ for all $1\le r\le p-1$,  with $K'_r=\cC(e^{\w t_r},1/4)$ and $K_r=\cC(e^{\w t_r},1/3)$. Hence \eqref{it:cone-condition-holds} is satisfied, and this concludes the proof.
\end{proof}

\section{Sub-additive thermodynamic formalism and proofs of Theorems \ref{thm:continuity-pressure}-\ref{thm:general-result}}
\label{sec:sub-proofs}

\subsection{Variational principle for sub-additive pressure}
\label{subsec:TF}

In order to prove our main results,  we require some elements from the sub-additive thermodynamic formalism.

As before, let $(X, T)$ be a subshift of finite type.  A sequence $\cF=\{\log f_n\}$ of functions on $X$ is said to be a \emdef{sub-additive potential} if
$$0\leq f_{n+m}(x)\le f_n(x)f_m(T^n x)$$ for all $x\in X$ and $n,m\in\N$. The \emdef{topological pressure} of a sub-additive potential $\cF$ (with respect to the shift $T$)  is defined as
$$P(T, \cF)= \lim_{n\to\infty} \frac{1}{n}\log\left( \sum_{\mathbf{i}\in X^*_n} \sup_{y\in [\mathbf{i}]} f_n(y) \right).
$$
The limit can be seen to exist by using a standard sub-additivity argument. We remark that we allow the functions $f_n$ to take the value $0$.

\begin{example}\begin{itemize}
\item[(i)] Given $g\in C(X)$, $s\geq 0$ and a matrix cocycle $A:x\to \R^{d\times d}$, define $\cF=\{\log f_n\}$ by
$
f_n(x)=\exp(S_ng(x))\phi^s(A(x,n))
$. Then $\cF$ is a sub-additive potential and $P(T, \cF)$ recovers the quantity $P_g(A,s)$ introduced in Section~\ref{sec:introduction}.
 \item[(ii)] Let $X=\{1,\ldots, k\}^\N$ be a full shift and  $A$ a locally constant matrix cocycle on $X$. For $s\geq 0$, set $\cF=\{\log f_n\}$ by
 $
f_n(x)=\|A(x,n)\|^s
$. Then $P(T, \cF)$ recovers the quantity $M(\bA,s)$ defined in Section~\ref{sec:introduction}.
\end{itemize}
\end{example}

Next we  prove a simple semi-continuity result.

\begin{lemma} \label{lem:general-semicontinuity}
Let $\mathcal{F}=\{\log f_n\}$, $\mathcal{G}^{(k)}=\{ \log g_n^{(k)}\}, k\in\N$, be sub-additive potentials on a common subshift of finite type $X$. Suppose that for each $n\in\N$, there exists a sequence of positive numbers $(\delta_k)\searrow 0$ (depending on $n$) such that $g_n^{(k)}(x)-f_n(x)\leq \delta_k$ for $x\in X$ and $k\in \N$. Then $\limsup_{k\to\infty} P(T, \mathcal{G}^{(k)}) \le P(T, \mathcal{F})$.
\end{lemma}
\begin{proof}
Let $n\in \N$. Then there exists a sequence of positive numbers $(\epsilon_k)\searrow 0$  such that
$$
\sum_{\mathbf{i}\in X^*_n} \sup_{y\in [\mathbf{i}]} g_n^{(k)}(y)\leq \epsilon_k+\sum_{\mathbf{i}\in X^*_n} \sup_{y\in [\mathbf{i}]} f_n(y).
$$
 Since the $\limsup$ in the definition of the pressure is an infimum, we have $$P(T,\mathcal{G}^{(k)})\le  (1/n)\log \left(\epsilon_k+\sum_{\mathbf{i}\in X^*_n} \sup_{y\in [\mathbf{i}]} f_n(y)\right).$$
Hence $\limsup_{k\to\infty} P(T, \mathcal{G}^{(k)})\leq (1/n)\log \sum_{\mathbf{i}\in X^*_n} \sup_{y\in [\mathbf{i}]} f_n(y)$. Letting $n\to \infty$, we obtain the desired inequality.
\end{proof}


 For $\mu \in\cE$, let $h_\mu$ denote the measure-theoretic entropy of $\mu$ (cf. \cite{Bowen75}).  Our proof of Theorem \ref{thm:general-result} depends on the following general variational principle for sub-additive potentials.
\begin{theorem}[\cite{CFH08}, Theorem 1.1] \label{thm:sub-additive-VP}
Let $(X,T)$ be a subshift of finite type and  $\mathcal{F}=\{\log f_n\}$  a sub-additive potential on $X$. Assume that  $f_n$ is continuous on $X$ for each $n$. Then
\begin{equation}
\label{variational-principle}
P(T, \mathcal{F}) = \sup\left\{ h_\mu + \lim_{n\to\infty} \frac{1}{n} \int \log(f_n(x)) d\mu(x)  : \mu\in\cE \right\}.
\end{equation}
\end{theorem}
Although in \cite{CFH08} this is proved for potentials on an arbitrary continuous dynamical system on a compact space, we state it only for subshifts of finite type. Particular cases of the above result, under stronger assumptions on the potentials, were previously obtained by many authors, see for example \cite{Falconer88b, Kaenmaki04, Mummert06, Barreira10} and references therein.

Measures that achieve the supremum in \eqref{variational-principle}  are called \emdef{ergodic equilibrium measures} for the potential $\cF$. The existence of ergodic equilibrium measures was proved by K\"{a}enm\"{a}ki \cite{Kaenmaki04} under fairly general conditions (but more restrictive than the above theorem). Nevertheless this existence result still holds  under the general setting of Theorem \ref{thm:sub-additive-VP}, thanks to the semicontinuity of entropy (see, e.g., \cite[Proposition 3.5]{Feng11} and the remark there). In this setting, equilibrium measures do not need to be unique.

\subsection{Proof of Theorem \ref{thm:general-result}} We first give a simple lemma.

\begin{lemma}
\label{SVF-inequality}
Let $B\in \R^{d\times d}$. Let $m, n$ be two integers with $0\leq m<n\leq d$. Then for any
$s\in [m,n]$, we have
$$
\phi^s(B)\geq \left(\phi^m(B)\right)^{\frac{n-s}{n-m}} \left(\phi^n(B)\right)^{\frac{s-m}{n-m}}.
$$
In particular, $\phi^s(B)\geq \left(\phi^n(B)\right)^{\frac{s}{n}}$.
\end{lemma}
\begin{proof}
Let $\alpha_1\geq \cdots\geq  \alpha_d$ be the square roots of the eigenvalues of $B^*B$, and let $p=\lfloor s\rfloor$. Then
$\phi^n(B)=\alpha_1\ldots \alpha_n\leq \alpha_1\ldots\alpha_p (\alpha_{p+1})^{n-p}$. It follows that
\begin{align*}
\frac{\left(\phi^s(B)\right)^{n-m}}{\left(\phi^m(B)\right)^{{n-s}} \left(\phi^n(B)\right)^{{s-m}}}&\geq \frac{(\alpha_1\cdots \alpha_p(\alpha_{p+1})^{s-p})^{n-m}}{\left(\alpha_1\cdots \alpha_m\right)^{n-s}(\alpha_1\cdots \alpha_p(\alpha_{p+1})^{n-p})^{s-m}}\\
 &=\frac{(\alpha_{m+1}\cdots \alpha_p)^{n-s}}{(\alpha_{p+1})^{(n-s)(p-m)}}\geq 1.
\end{align*}
That is, $
\phi^s(B)\geq \left(\phi^m(B)\right)^{\frac{n-s}{n-m}} \left(\phi^n(B)\right)^{\frac{s-m}{n-m}}
$. Taking $m=0$, we obtain $\phi^s(B)\geq \left(\phi^n(B)\right)^{\frac{s}{n}}$.
\end{proof}

By the concavity of $s\to \log \phi^s(B)$ over $[0,d]$, the above lemma can be extended to non-integer $m, n$, but we will not require this.

\begin{proof}[Proof of Theorem \ref{thm:general-result}(1)]
If $s\geq d$,  $P_{g}(A,s)$ can be viewed as the classical topological pressure (in the additive setting) of the potential
\[
g(x)\log|\det(A(x))|^{s/d}.
\]
Since this potential is uniformly continuous in $A$ (in the $L^\infty$ topology) the continuity of the map $A\to P_{g}(A,s)$ follows by a standard argument. Hence in the following we assume that $0<s< d$.

It follows from an application of Lemma \ref{lem:general-semicontinuity} that $P_g(\cdot,s)$ is upper semi-continuous at $A$, so we only need to prove the lower semi-continuity. Let $A:X\to\R^{d\times d}$ be locally constant and $s>0$. If $P_{g}(A,s)=-\infty$, there is nothing to prove, so assume that $P_{g}(A,s)>-\infty$.

Fix $\e>0$. By the variational principle (Theorem \ref{thm:sub-additive-VP}) and subadditivity, there is $\mu\in\cE$ such that
\begin{equation}
P_{g}(A,s) -\e \le h_\mu + \int g d\mu+E_\mu(A,s), \label{eq:variational-principle}
\end{equation}
where
\[
E_\mu(A,s) = \lim_{n\to\infty}\frac{1}{n}\int \log \phi^s(A(x,n))\, d\mu(x).
\]

 Let $\lambda_i,d_i, t_i$ ($1\leq i\leq p$) be as in Theorem \ref{thm:sub-cocycle-cone-condition}. For convenience set $t_0=0$.  There is a unique $r\in \{0,1,\ldots, p-1\}$ such that $t_r < s\le t_{r+1}$. Note that $\lambda_{r+1}>-\infty$, for otherwise we would have $E_\mu(A,s)=-\infty$ and hence $P_{g}(A,s)=-\infty$, contrary to our assumption.

 Let $\eta, S$ and $\{\Sigma_n:n\in S\}$ be as in Theorem \ref{thm:sub-cocycle-cone-condition}. Write $\Gamma_m = \sum_{i=1}^m d_i\lambda_i$ for $m=1,\ldots,  p$ and set $\Gamma_0=0$ for   convention.
Recall from Lemma \ref{lem:oseledets-gives-Lyapunov-for-SV} that
\begin{equation} \label{eq:lyapunov-exponents-relation}
E_\mu(A,s)= \Gamma_r + (s-t_r)\lambda_{r+1}.
\end{equation}

By Egorov's Theorem, the Shannon-McMillan-Breiman Theorem, and the sub-additive ergodic Theorem,  there exist $n\in S$ and a set $\Delta_n\subset \Sigma_n\subset X^*_n$, such that
\begin{equation}\label{eq:large-size}
\sum_{\mathbf{i}\in\Delta_n} \mu[\mathbf{i}] > \eta/2
\end{equation}
and, if $x|n\in\Delta_n$, then
\begin{align}
S_ng(x) &> n(\mu(g)-\e),\label{eq:g-close-to-lyapunov-exp}\\
\mu[x|n] &< \exp(n (\e-h_\mu)),\label{eq:small-measure}\\
  \exp(n(\Gamma_m-\e)) &< \phi^{t_m}(A(x,n)) < \exp(n(\Gamma_m+\e))\quad \mbox{ for  }m=1, \ldots, p-1,\label{eq:large-SV-1}\\
\phi^{t_{p}}(A(x,n)) &>  \exp(n(\Gamma_{p}-\e)).\label{eq:large-SV-2}
\end{align}

It follows from \eqref{eq:large-size} and \eqref{eq:small-measure} that
\begin{equation} \label{eq:lower-bound-size}
\#\Delta_n \ge \frac{\eta}{2} \exp(n(h_\mu-\e)).
\end{equation}
Denote by $\Delta_n^\ell\subset X^*_{\ell n}$ the family of juxtapositions of $\ell$ words in $\Delta_n$.

First assume that  $p=1$. In this case, the cocycle $A$ has a single Lyapunov exponent $\lambda=\lambda_1$ with respect to $\mu$.   It follows from Lemma \ref{lem:oseledets-gives-Lyapunov-for-SV} that $E_\mu(A,s)=s\lambda$. Meanwhile, since $t_1=d$ in this case, we have  $\phi^{t_1}(M)=|\det(M)|$ for any $M\in \R^{d\times d}$.  As
 $\Delta_n$ is finite, we can find a $L^\infty$ neighborhood $\mathcal{U}$ of $A$ in $\cM(X, d)$ such that if $B\in\mathcal{U}$, then
\[
\phi^d(B(x, n)) \ge e^{-n\e} \phi^d(A(x,n))\geq e^{nd(\lambda-2\e) }, \mbox{ if } x|n\in \Delta_n.
\]
Therefore
\[
\phi^d(B(x,kn)) \ge e^{\ell nd(\lambda-2\e)}, \mbox{ if }x|kn\in \Delta_n^\ell.
\]
By Lemma \ref{SVF-inequality}, we have
\[
\phi^s(B(x, \ell n)) \ge \left(\phi^d(B(x,kn))\right)^{\frac{s}{d}}\geq e^{\ell ns(\lambda-2\e)}\quad\text{if }x|\ell n\in \Delta_n^\ell.
\]
Given $x\in X^*$, let $\overline{x}\in X$ be any infinite word starting with $x$. With this information, and also using \eqref{eq:variational-principle}, \eqref{eq:g-close-to-lyapunov-exp} and \eqref{eq:lower-bound-size}, we can conclude:
\begin{align}
P_{g}(B,s) &\ge \limsup_{\ell\to\infty} \frac{1}{n\ell} \log\left(\sum_{x\in \Delta_n^\ell}\exp(S_{\ell n}g(\overline{x})) \phi^s(B(\overline{x},\ell n))\right)\nonumber \\
&\geq \mu(g)+h_\mu+s\lambda-((2+2s)\e+C/n) \nonumber\\
&\ge  P_g(A,s) -((3+2s)\e+C/n)  \nonumber
\end{align}
for some constant $C>0$ independent of $B$, $\mathcal{U}$, and $n$.  Since $(3+2s)\e+C/n$ can be made arbitrarily small, this establishes the lower semicontinuity for the case $p=1$.

Next we assume that $p>1$.
 It follows from Theorem \ref{thm:sub-cocycle-cone-condition} (and the fact that $\Delta_n\subset\Sigma_n$) that the families $\{ A^{\w t_m}(y,n): y|n \in \Delta_n\}$ satisfy the cone condition with some cones $\overline{K_m},\overline{K_m}'$ for $m=1,\ldots, p-1$. By the robustness of the cone condition, there is an $L^\infty$ neighborhood $\cU$ of $A$ such that if $B:X\to\R^{d\times d}\in\cU$, then $\{ B^{\w t_m}(y,n) :y|n \in \Delta_n\}$ satisfy the cone condition with cones $K_m, K'_m$ (obtained from perturbing $\overline{K_m},\overline{K_m}'$ slightly), which do not depend on $B$, $n$, or the particular neighborhood (as long as it is small enough). Since the conditions \eqref{eq:large-SV-1} (for $m=1,\ldots, p$)  are $L^\infty$-open, by making $\mathcal{U}$ smaller we can ensure that they continue to hold for all $B\in\mathcal{U}$ in place of $A$.

We know from \eqref{eq:large-SV-1} and Lemma \ref{lem:cone-implies-almost-multiplicative} that, if $B\in\mathcal{U}$ and $x|\ell n\in\Delta_n^\ell$, then
\begin{align}
 c^\ell \exp\left(n\ell(\Gamma_m-\e)\right) &\le \phi^{t_m}(B(x,{ \ell }n)) \le \exp\left(n\ell(\Gamma_m+\e)\right)\mbox { for }m=1,\ldots, p-1,\label{eq:large-SV-B-1}
 \end{align}
where $c>0$ is independent of $n$ and $B$.
  As $t_p=d$ and $\phi^d(M)=|\det(M)|$ for $M\in \R^{d\times d}$, by \eqref{eq:large-SV-2} we also have
 \begin{equation}
\phi^{t_{p}}(B(x,{ \ell } n)) \ge \exp\left(n\ell(\Gamma_{p}-\e)\right)\label{eq:large-SV-B-1'}
\end{equation}
for $B\in\mathcal{U}$ and $x|\ell n\in\Delta_n^\ell$. Since $t_r<s\leq t_{r+1}$, by Lemma \ref{SVF-inequality}, \eqref{eq:large-SV-B-1} and \eqref{eq:large-SV-B-1'}
\begin{align} \label{eq:large-SV-B-2}
\phi^s (B(x,\ell n)) &\ge \left(\phi^{t_{r}}(B(x,{\ell }n)) \right)^{\frac{t_{r+1}-s}{d_{r+1}} }  \left( \phi^{t_{r+1}}(B(x,{\ell }n))\right)^{\frac{s-t_r}{d_{r+1}}}\nonumber\\
&\ge c^\ell \exp(n\ell [\Gamma_r+(s-t_r)\lambda_{r+1}-2\e]).
\end{align}
Putting together \eqref{eq:variational-principle}, \eqref{eq:g-close-to-lyapunov-exp},  and \eqref{eq:large-SV-B-2},  we estimate (always assuming $B\in\mathcal{U}$, $x|n\ell\in \Delta_n^\ell$)
\begin{equation}\label{eq:large-SV-B-3}
\exp(S_{\ell n}g(x))\phi^s(B(x,{\ell }n)) \ge c^{2\ell}\exp\left( n\ell\left( P_g(A,s) - h_\mu - 4\e\right) \right).
\end{equation}


Given $x\in X^*$, let $\overline{x}\in X$ be any infinite word starting with $x$. With this information, and also using \eqref{eq:lower-bound-size}, we can conclude:
\begin{align}
P_{g}(B,s) &\ge \limsup_{\ell\to\infty} \frac{1}{n\ell} \log\left(\sum_{x\in\Delta_n^\ell}\exp(S_{\ell n}g(\overline{x})) \phi^s(B(\overline{x},\ell n))\right)\nonumber \\
&\ge \limsup_{\ell\to\infty} \frac{1}{n\ell} \log\left((\eta/2)^\ell \exp(n\ell(h_\mu-\e)) \cdot c^{2\ell} \exp\left( n\ell\left( P_g(A,s) -h_\mu-4\e\right) \right)\right) \nonumber\\
&\ge P_{g}(A,s) - \left(5\e+C/n\right)  \label{eq:large-SV-B-4},
\end{align}
for some constant $C>0$ independent of $B$, $\mathcal{U}$, and $n$. Since $5\e+C/n$ can be made arbitrarily small, this establishes the lower semicontinuity for the case $p>1$.
 This completes the proof of  Theorem \ref{thm:general-result}(1).
\end{proof}

\begin{proof}[Proof of Theorem \ref{thm:general-result}(2)]
We first prove that the map $P(A, s)$ is  jointly continuous at $(A,s)$ when $A$ is locally constant and $s\not\in  \{0, 1,\ldots, d-1\}$.
Here we only do it in the case $1\leq r<p-1$ (the cases $r=0$ and $r=p-1$ can be dealt with in a similar way). Notice that  for $B\in\mathcal{U}$,  $x|\ell n\in\Delta_n^\ell$ and $t_r<s'\leq t_{r+1}$, instead of \eqref{eq:large-SV-B-2}-\eqref{eq:large-SV-B-4}, we can prove similarly
\begin{align*}
&\phi^{s'}(B(x,{ \ell }n))
\ge c^\ell \exp(n\ell [\Gamma_r+(s'-t_r)\lambda_{r+1}-2\e]),\\
&\exp(S_{\ell n}g(x))\phi^{s'}(B(x,{\ell }n)) \ge c^{2\ell}\exp\left( n\ell\left( P_g(A,s) +(s'-s)\lambda_{r+1}- h_\mu - 4\e\right) \right),\\
&P_{g}(B,s')\ge P_{g}(A,s) + (s'-s)\lambda_{r+1}- \left(5\e+C/n\right).
\end{align*}
This proves the lower semi-continuity (and hence the continuity) of the pressure map at the point $(A,s)$ if  $s\not\in \{0, 1,\ldots, d-1\}$ (since in such case $t_r<s< t_{r+1}$).

 To complete the proof of part (2), assume that  $A$ is a locally constant matrix cocycle taking values in $GL_d(\R)$ and $s\geq 0$. When the neighborhood $\mathcal{U}$ of $A$ is taken small enough, there exist two positive constant $C, D>0$ such that for any $B\in \mathcal{U}$, $x\in X$,
    $$
 D<\frac{1}{\|B(x)^{-1}\|}\leq \|B(x)\|<C.
 $$
Hence for $n\in \N$,
   $$
 D^n<\frac{1}{\|B(x,n)^{-1}\|}\leq \|B(x,n)\|<C^n.
 $$
 Therefore all the singular values of $B(x,n)$ are between $D^n$ and $C^n$, from which we can easily deduce that
\[
D^{(s'-s)n}\leq \frac{\phi^{s'}(B(x,n))}{\phi^{s}(B(x,n))}\leq C^{(s'-s)n}
\]
if $s'>s$; and
\[
C^{(s'-s)n}\leq \frac{\phi^{s'}(B(x,n))}{\phi^{s}(B(x,n))}\leq D^{(s'-s)n}
\]
if $s'\leq s$. This implies that
\begin{equation}
\label{eq-close}
 |P_{g}(B,s')-P_g(B,s)|\leq |s-s'|\cdot \max\{|\log C|, |\log D|\}.
 \end{equation}
Recall that we have shown in part (1) that $P_g(B,s')$ tends to  $P(A,s')$ if $B\to A$ and $s'>0$. Due to this fact and \eqref{eq-close}, we see that $P_g(B,s')$ tends to  $P(A,s)$ if $(B,s')\to (A,s)$.
 This finishes the proof of the theorem.
 \end{proof}

\subsection{ Proofs of Theorems \ref{thm:continuity-pressure} and \ref{thm:matrix-pressure-continuity}}

\begin{proof}[Proof of Theorem \ref{thm:continuity-pressure}]  Statements (1) and (3)  follow directly from Theorem \ref{thm:general-result}  by letting $X$ be a full shift, and $g\equiv 0$.

To prove (2), let $\bA\in \cA_{d, k}^C$ and $\epsilon>0$ (the case $\e=0$ can be handled easily). Then $P(\bA, s(\bA)+\epsilon)<0$. By (1), when $\bB$ is close enough to $\bA$, we have $P(\bB,  s(\bA)+\epsilon)<0$ and hence $s(\bB)\leq   s(\bA)+\epsilon$. Since $\epsilon$ can be taken arbitrary small, this proves $\limsup_{\bB\to \bA}s(\bB)\leq s(\bA)$. Next we prove
$\liminf_{\bB\to \bA}s(\bB)\geq s(\bA)$. For this purpose, we can assume that $s(\bA)>0$; otherwise we have nothing left to prove. Then  $P(\bA, s(\bA)-\epsilon)>0$ for $0<\epsilon<s(\bA)$. Applying (1) again, we see that when $\bB$ is close enough to $\bA$,  $P(\bB, s(\bA)-\epsilon)>0$ and thus $s(\bB)>s(\bA)-\epsilon$. This proves $\liminf_{\bB\to \bA}s(\bB)\geq s(\bA)$, and we are done.
\end{proof}

\begin{proof}[Proof of Theorem \ref{thm:matrix-pressure-continuity}]We follow essentially the same idea used in the proof of Theorem \ref{thm:general-result}. Again, upper semi-continuity follows from Lemma  \ref{lem:general-semicontinuity}, so we only need to prove the lower semi-continuity of $M(\cdot,\cdot)$.

Let $s\geq 0$ and $\bA\in \cA_{d,k}$.
Let $(X, T)$ be the full shift  over the alphabet $\{1,\ldots, k\}$ and $A: X\to \R^{d\times d}$ be the matrix cocycle generated by $\bA$.  Fix $\e>0$. By the variational principle (Theorem \ref{thm:sub-additive-VP}), there is $\mu\in\cE$ such that
\begin{align}
M(\bA,s) -\e &\le h_\mu + s\lim_{n\to\infty}\frac{1}{n}\int \log \|A(x,n)\|\, d\mu(x)\nonumber \\
&=: h_\mu  + s\lambda_1(\mu). \label{eq:variational-principle'}
\end{align}
Let $\lambda_i,d_i, t_i$ ($1\leq i\leq p$), $\eta, S$ and $\{\Sigma_n:n\in S\}$ be as in Theorem \ref{thm:sub-cocycle-cone-condition}.  Write $\Gamma_m = \sum_{i=1}^m d_i\lambda_i$ for $m=1,\ldots,  p$ and set $\Gamma_0=0$ for   convenience. Clearly $\lambda_1(\mu)=\lambda_1$. We may assume  $\lambda_{1}>-\infty$, for otherwise we would have $M(\bA,s)=-\infty$ and the lower semi-continuity of $M(\cdot, \cdot)$ at $(\bA,s)$ follows automatically.

As proved in  Theorem \ref{thm:sub-cocycle-cone-condition},
there is $n\in S$ and a set $\Delta_n\subset \Sigma_n$, such that \eqref{eq:large-size}, \eqref{eq:small-measure}, \eqref{eq:large-SV-1} and \eqref{eq:lower-bound-size} hold; in particular, there is an open neighborhood ${\mathcal U}$ of $A$ in $\cM(X,d)$ such that \eqref{eq:large-SV-B-1} and \eqref{eq:large-SV-B-1'} hold for   $B\in\mathcal{U}$ and $x|\ell n\in\Delta_n^\ell$, where
 $\Delta_n^\ell\subset \{1,\ldots, k\}^{\ell n}$ denotes the family of juxtapositions of $\ell$ words in $\Delta_n$. By Lemma \ref{SVF-inequality}, \eqref{eq:large-SV-B-1} and \eqref{eq:large-SV-B-1'}, we have for   $B\in\mathcal{U}$ and $x|\ell n\in\Delta_n^\ell$,
 $$
 \|B(x,\ell n)\|=\phi^{ 1}(B(x,\ell n))\geq \left(\phi^{ t_1} (B(x,\ell n))\right)^{\frac{s}{t_1}}\geq c^\ell \exp(\ell n(\lambda_1-\epsilon)),
 $$
and hence for $s'\geq 0$,
 $$
 \|B(x,\ell n)\|^{s'}\geq c^{\ell s'} \exp(s'\ell n(\lambda_1-\epsilon)).
 $$
In particular,  this holds for all locally constant cocycles $B\in \mathcal{U}$, which we identify with elements $\bB\in\cA_{d,k}$.  Given $x\in X^*$, let $\overline{x}\in X$ be any infinite word starting with $x$.   With this information, and also using \eqref{eq:lower-bound-size}, we can conclude:
 \begin{align*}
M(\bB,s') &\ge \limsup_{\ell\to\infty} \frac{1}{n\ell} \log\left(\sum_{x\in\Delta_n^\ell}\|B(\overline{x},\ell n)\|^{s'}\right)\nonumber \\
&\ge \limsup_{\ell\to\infty} \frac{1}{n\ell} \log\left((\eta/2)^\ell \exp(n\ell(h_\mu-\e)) \cdot c^{\ell s'} \exp\left( n\ell s'(\lambda_1-\e) \right)\right) \nonumber\\
&\ge M(\bA,s) +(s'-s)\lambda_1+(s'+1) \left(2\e+C/n\right)  
\end{align*}
for some constant $C>0$ independent of $\bB$, $\mathcal{U}$, $s'$ and  $n$. Since $2\e+C/n$ can be made arbitrarily small, this establishes the lower semicontinuity.
\end{proof}

\subsection{An alternative approach in the two-dimensional case}

A result of Bocker and Viana \cite{BockerViana10} on continuity of Lyapunov exponents for IID $2\times 2$ invertible matrix cocycles can be
used to give a short alternative proof of Theorems \ref{thm:continuity-pressure} and \ref{thm:matrix-pressure-continuity} in the case that $d = 2$ and $\bA\in \cG_{d,k}$.

To state the result of Bocker and Viana, let $k\geq 2$ be a positive integer and  $\Omega_k$ denote the collection of strictly positive probability vectors in $\R^k$. Let $X$ denote the full shift space over $k$ symbols. For ${\bf p}=(p_1,\ldots, p_k)\in \Omega_k$, let $\mu_{\bf p}$ denote the Bernoulli product measure $\prod_{n=1}^\infty (p_1,\ldots, p_k)$ on $X$. The main result of Bocker
and Viana in \cite{BockerViana10} can be formulated as follows. Recall that for $\bA\in\cG_{d,k}$ and $s\ge 0$ we write
\[
 E_\mu(\bA,s) = \lim_{n\to\infty}\frac1n \int \phi^s(A(n,x)) \, d\mu(x),
\]
where $A$ is the associated cocycle.

\begin{thm}
\label{thm-1} The map $(\bA, {\bf p})\to E_{\mu_{\bf p}}(\bA, 1)$ is continuous over $\cG_{d,k}\times \Omega_k$.
\end{thm}

Without loss of generality and for brevity, in the following we only show that the above theorem can be used to prove the continuity of  the map $(\bA,s)\to P(\bA,s)$  over $\cG_{2,k}\times (1,2)$. The full statements of Theorems 1.2 and 1.3 (for $d = 2$ and $\bA\in \cG_{d,k}$) can be proved by using nearly identical arguments.

Due to a result of Falconer and Sloan \cite[Corollary 1.3]{FalconerSloan09}, we only need to prove that $(\bA, s)$ is a continuity point of $P(\cdot, \cdot)$ if $\bA=(A_1,\ldots, A_k) \in \cG_{2,k}$ is reducible (in the sense that the $A_i$ have a common eigenvector in $\R^2$).

Now fix a reducible cocycle $\bA=(A_1,\ldots, A_k) \in \cG_{2,k}$, and $s\in (1,2)$. Assume that $v\in \R^2$ is a common eigenvector of $A_i$. Pick $T\in GL_2(\R)$ so that
$T(1,0)^*=v$. Then $B_i:=T^{-1}A_iT$ are upper triangular matrices, say $\left(\begin{array}{cc}
a_i & c_i\\
0& b_i
\end{array}
\right)
$.  Denote $\bB=(B_1,\ldots, B_k)$. Let $f, g: X\to \R$ be locally constant functions, defined respectively by
$$
f(x)=\log a_{x_1}+(s-1)\log b_{x_1},\quad g(x)=\log b_{x_1}+(s-1)\log a_{x_1}.
$$
Following the same arguments as those in the proof of \cite[Theorem 1.7(i)]{FengKaenmaki11}, we can prove that  for any $\mu\in \cE$,

$$
E_\mu(\bB, s)=\max\{\mu(f), \mu(g)\}.
$$
Noticing that $E_\mu(\bA,s)=E_\mu(\bB, s)$,
by the variational principle (Theorem \ref{variational-principle}), we have
$$
P(\bA,s)=\max\left\{\sup_{\mu\in \cE} (h_\mu+\mu(f)),\quad \sup_{\mu\in \cE} (h_\mu+\mu(g))\right \}.
$$
Since $f$ and $g$ are locally constant, we see that that the set $\cE_0(\bA, s)$ of ergodic equilibrium measures consists of one or two Bernoulli product measures on $X$. Take $\mu_{\bf p}\in \cE_0(\bA, s)$. For any $\bA'\in \cG_{2,k}$ and $s'\in (1,2)$, using  the variational principle again we have
\begin{align}
P(\bA',s')&\geq h_{\mu_{\bf p}}+E_{\mu_{\bf p}}(\bA', s')\nonumber\\
&=h_{\mu_{\bf p}}+(2-s')E_{\mu_{\bf p}}(\bA', 1)+(s'-1)E_{\mu_{\bf p}}(\bA', 2). \label{e-1}
\end{align}
Notice that $E_{\mu_{\bf p}}(\bA', 2)=\sum_{i=1}^k p_i \log |\det(A_i')|$. Hence when $\bA'$ tends to $\bA$, we have
$$
E_{\mu_{\bf p}}(\bA', 2)\to E_{\mu_{\bf p}}(\bA, 2) \
\mbox{ and }E_{\mu_{\bf p}}(\bA', 1)\to E_{\mu_{\bf p}}(\bA, 1).
$$
where the second convergence follows from Theorem \ref{thm-1}.  Hence by \eqref{e-1}, we have
$$
\liminf_{(\bA',s')\to (\bA,s)}P(\bA',s')\geq h_{\mu_{\bf p}}+E_{\mu_{\bf p}}(\bA, s)=P(\bA,s).
$$
This proves the lower semi-continuity (and hence the continuity) of $P(\cdot,\cdot)$ at $(\bA,s)$.

\section{Further generalizations and applications}
\label{sec:generalizations}
\subsection{A further generalization of Theorems \ref{thm:matrix-pressure-continuity}-\ref{thm:general-result}} For $s_1\ge \cdots \ge s_d\geq 0$, we define the {\em  generalized singular value function} $\phi^{s_1,\ldots, s_d}:\R^{d\times d}\to [0,\infty)$ as
\begin{equation}
\label{eq: generalizedSVF}
\phi^{s_1,\ldots, s_d} (A) = \alpha_1(A)^{s_1}\cdots\alpha_d(A)^{s_d}=\left(\prod_{m=1}^{d-1}\|A^{\w m}\|^{s_m-s_{m+1}}\right)\|A^{\w d}\|^{s_d}.
\end{equation}
When $s\in [0,d]$, the singular value function $\phi^s(\cdot)$ coincides with the generalized singular value function $\phi^{s_1,\ldots, s_d}(\cdot)$, where
$$(s_1,\ldots,  s_d)=(\underbrace{1,\ldots, 1}_{m \text{ times}}, s-m, 0,\ldots, 0),$$ with $m=\lfloor s\rfloor$. The generalized singular value function also arises in connection with the $L^q$ spectrum of measures on self-affine sets, see \cite{Falconer99} and Section~\ref{subsec:further-applications} below.

From the second equality in \eqref{eq: generalizedSVF},  we see that the generalized singular value function is sub-multiplicative. Let $(X, T)$ be a subshift of finite type. For  $g\in C(X)$ and $A\in \cM(X,d)$,  let $P_g(A, (s_1,\ldots, s_d))$ denote the topological pressure of the sub-additive potential
 $\cF=\{\log f_n\}$ with
\begin{equation} \label{eq:generalized-potential}
f_n(x)=\exp(S_ng(x)) \phi^{s_1,\ldots, s_d}(A(x,n)).
\end{equation}

Denote $\Lambda_d:=\{(s_1,\ldots, s_d): s_1\geq \cdots\geq s_d\geq 0\}$.
As a generalization of Theorems \ref{thm:matrix-pressure-continuity}-\ref{thm:general-result}, we have
\begin{thm}
\label{thm:generalization}
 \begin{itemize}
\item[(1)]For fixed  $(s_1,\ldots, s_d)\in \Lambda_d$, any  locally constant cocycle $A$ is a continuity point of the pressure map $B\to P_g(B,(s_1,\ldots, s_d))$ on $\cM(X, d)$.
 \item[(2)]  For each locally constant cocycle $A\in \cM(X, d)$ taking values in $GL_d(\R)$ and for each $(s_1,\ldots, s_d)\in \Lambda_d$,  the pair $(A,(s_1,\ldots, s_d))$ is a continuity point of the pressure map $P_g(\cdot,\cdot)$ on $\cM(X, d)\times \Lambda_d$.
 \end{itemize}
\end{thm}
The proof is nearly identical to that of Theorems \ref{thm:matrix-pressure-continuity}-\ref{thm:general-result} so is omitted.

\begin{remark}
It is also possible to generalize this theorem further as follows: in \eqref{eq:generalized-potential}, $\exp(S_n g(x))$ can be replaced by $g_n(x)$, where $\mathcal{G}=\{ \log g_n \}$ is an \emph{almost-additive} sequence of potentials, that is, $C^{-1} g_{n+m}(x) \le g_n(x) g_m(T^n x) \le C g_{n+m}(x)$ for all $x\in X$ and $n,m\in\N$, where $C>0$ is independent of $n,m,x$. Again, the proof is the same with routine changes.
\end{remark}

\subsection{An application to the continuity of equilibrium states}

Let $\mathcal{F}=\{ \log f_\ell\}$ be a sub-additive potential on a subshift of finite type $X$. For any $T$-invariant measure $\mu$ (not necessarily ergodic),
\[
P(T, \mathcal{F}) \ge h_\mu + \lim_{\ell\to\infty} \frac{1}{\ell} \int \log(f_\ell(x)) d\mu(x).
\]
For the proof, see e.g. \cite{CFH08}. Invariant measures for which there is equality above are called \emph{equilibrium measures} (or equilibrium states) for the potential $\mathcal{F}$. We recall from our discussion in Section~\ref{subsec:TF} that equilibrium measures always exist, but they do not need to be unique. It is natural to ask if the equilibrium measures vary continuously with the potential; we give a partial answer for the potentials considered in this article. Let $X$ be a subshift of finite type and fix $g\in C(X)$.  For a cocycle $A\in\cM(X,d)$ and $(s_1,\ldots, s_d)\in\Lambda_d$, let $\mathcal{E}_g(A,(s_1,\ldots,s_d))$ be the collection of equilibrium measures for the potential given by \eqref{eq:generalized-potential}. Furthermore, set
\begin{equation} \label{eq:def-generalized-energy}
E_\mu(A,(s_1,\ldots,s_d)) =  \lim_{n\to\infty} \frac{1}{n} \int \log \phi^{s_1,\ldots,s_d}(A(x,n))\, d\mu(x).
\end{equation}

\begin{proposition} \label{prop:continuity-eq-state}

\begin{itemize}
 \item[(1)] Let $(s_1,\ldots,s_d)\in\Lambda_d$, and $A_\ell, A\in\cM(X,d)$ be continuous cocycles such that $A_\ell\to A$ uniformly and $A$ is locally constant. If $\mu_\ell\to\mu$ weakly and $\mu_\ell\in \mathcal{E}_g(A_\ell,(s_1,\ldots,s_d))$, then $\mu\in\mathcal{E}_g(A,(s_1,\ldots,s_d))$. Moreover, $h_{\mu_\ell}\to h_\mu$ and $E_{\mu_\ell}(A_\ell,(s_1,\ldots,s_d))\to E_\mu(A,(s_1,\ldots,s_d))$.
 \item[(2)] Let $(s_1^{(\ell)},\ldots,s_d^{(\ell)})$ be a sequence in $\Lambda$ converging to $(s_1,\ldots,s_d)\in\Lambda$. Also, let $A_\ell, A\in\cM(X,d)$ be continuous cocycles such that $A_\ell\to A$ uniformly, and $A$ is locally constant and takes values in $GL_d(\R)$.  If $\mu_\ell\to\mu$ weakly and $\mu_\ell\in \mathcal{E}_g(A_\ell,(s_1^{(\ell)},\ldots,s_d^{(\ell)}))$, then $\mu\in\mathcal{E}_g(A,(s_1,\ldots,s_d))$. Moreover, $h_{\mu_\ell}\to h_\mu$ and $E_{\mu_\ell}(A_\ell,(s_1^{(\ell)},\ldots,s_d^{(\ell)}))\to E_\mu(A,(s_1,\ldots,s_d))$.
\end{itemize}
\end{proposition}
\begin{proof}
 By definition of equilibrium measure,
 \[
  P_g(A, (s_1,\ldots, s_d)) = h_{\mu} + \int g \,d\mu  +E_\mu(A,(s_1,\ldots,s_d)).
 \]
 Since the limit in \eqref{eq:def-generalized-energy} is in fact an infimum by sub-additivity, the function $(B,\nu)\to E_\nu(B,(s_1,\ldots,s_d))$ is upper semicontinuous, as it is an infimum of continuous functions (here we are considering the uniform topology on $C(X,\R^{d\times d})$ and the weak topology on the probability measures on $X$). Also, the entropy map $\nu\to h_\nu$ is upper-semicontinuous on any subshift (see e.g. \cite[Theorem 8.2]{Wal82}). Therefore, if $\mu_\ell\to\mu$ weakly, we conclude from these observations and the first part of Theorem \ref{thm:generalization} that
 \begin{align*}
 P_g(A, (s_1,\ldots, s_d)) &=  \lim_{\ell\to\infty} P_g(A_\ell, (s_1,\ldots, s_d)) \\
 &= \lim_{\ell\to\infty} h_{\mu_\ell} + \int g \,d\mu_\ell  +E_{\mu_\ell}(A_\ell,(s_1,\ldots,s_d))\\
 &\le h_\mu + \int g\,d\mu+ E_{\mu}(A,(s_1,\ldots,s_d)).
 \end{align*}
 This shows that $\mu\in \mathcal{E}_g(A,(s_1,\ldots,s_d))$ and there is equality throughout, giving the first claim. The second follows in the same way, using the second part of Theorem \ref{thm:generalization}.
\end{proof}

\begin{remark}
One could ask if the entire set of equilibrium measures varies continuously with the cocycle. However, there are simple counterexamples. For example, let $X$ be the full shift on two symbols, let
\[
A(0) = \left(
        \begin{array}{cc}
          2 & 0 \\
          0 & 1 \\
        \end{array}
      \right), \quad
A_\e(1) = \left(
        \begin{array}{cc}
          1+\e & 0 \\
          0 & 2 \\
        \end{array}
      \right),
\]
and let $A_\e$ be the locally constant cocycle taking values $A(0), A_\e(1)$. Then it follows either from a direct analysis (that we skip) or from \cite[Theorem 1.7]{FengKaenmaki11} that $\mathcal{E}_0(A_\e,(1,0,\ldots,0))$ is a singleton if $\e>0$, but $\mathcal{E}_0(A_0,(1,0,\ldots,0))$ contains two ergodic measures (and therefore also the segment joining them).
\end{remark}

\subsection{Continuity of the Lyapunov spectrum}

Let $X$ be a subshift of finite type over the alphabet $\{1,\ldots, k\}$. Let $\bA=(A_1,\ldots, A_k)\in \cG_{d, k}$. For $\alpha\in \R$, define
\[
\Delta_\bA(\alpha)=\left\{x\in X:\; \lim_{n\to \infty}\frac{1}{n}\log \|A(x,n)\|=\alpha \right\}.
\]
The map $\alpha\to h_{{\rm top}}(\Delta_\bA(\alpha))$, where $h_{{\rm top}}$ denotes the Bowen topological entropy for non-compact sets \cite{Bowen73}, is called the \emph{(upper) Lyapunov spectrum} of $\bA$, and a natural question is how it varies with $\bA$. The Lyapunov spectrum is closely related to the pressure function $M(\bA,q)$. Indeed, it follows from a more general result of Feng and Huang \cite[Theorem 1.3]{FengHuang10} that, for fixed $\bA\in\cG_{d,k}$, if $q>0$ and $\alpha$ is either the left or right derivative of $M(\bA,\cdot)$ at $q$ (which exist since $M(\bA,\cdot)$ is convex), then
\begin{equation} \label{eq:MF-Lyapunov}
h_{{\rm top}}(\Delta_\bA(\alpha)) = M(\bA,q)-\alpha q.
\end{equation}
From this we can deduce the following continuity result for the Lyapunov spectrum:
\begin{proposition} \label{prop:continuity-Lyap-spec}
Suppose $\bA_\ell,\bA\in\cG_{d,k}$ with $\bA_\ell\to \bA$ and $q_\ell,q >0$ such that $q_\ell\to q$. Moreover, assume that $P(\bA,\cdot)$ is differentiable at $q$, and let $\alpha$ be the derivative.

If $\alpha_\ell$ equals either the left or right derivative of $P(\bA_\ell,\cdot)$ at $q_\ell$, then
\[
\lim_{\ell\to\infty} h_{{\rm top}}(\Delta_{\bA_\ell}(\alpha_\ell)) =  h_{{\rm top}}(\Delta_{\bA}(\alpha)).
\]
\end{proposition}
\begin{proof}
By \eqref{eq:MF-Lyapunov}, we only need to prove that
\[
\lim_{\ell\to\infty} M(\bA_\ell,q_\ell)-\alpha_\ell q_\ell = M(\bA,q)-\alpha q.
\]
In light of Theorem \ref{thm:matrix-pressure-continuity}, all we need to show is that $\alpha_\ell\to\alpha$. By \cite[Theorem 3.3]{FengHuang10}, for any $\ell$ there exists an equilibrium measure $\mu_\ell$ for the potential $\mathcal{F}_\ell=\{ q_\ell\log\|A_\ell(x,n)\|\}_n$ such that
\[
\alpha_\ell = \lim_{n\to\infty} \frac1n \int  \log\| A_\ell(x,n)\|\, d\mu_\ell.
\]
Given any sub-sequence $\ell_j$, let $\mu$ be an accumulation point of $\mu_{\ell_j}$. Applying the second part of Proposition \ref{prop:continuity-eq-state} (with $g\equiv 0$, $(s_1,\ldots,s_d)=(q,0,\ldots,0)$ and $(s_1^{(\ell)}, \ldots, s_d^{(\ell)})=(q_\ell, 0,\ldots, 0)$), we deduce that $\mu$ is an equilibrium measure (for $\mathcal{F}=\{ q\log\|A(x,n)\|\}_n$), and moreover $(\alpha_{\ell_j})$ accumulates to
\[
\lim_{n\to\infty} \frac1n \int \log\| A(x,n)\|\, d\mu.
\]
However, this expression equals $\alpha$ for any equilibrium measure $\mu$, thanks to the differentiability of $M(\bA,\cdot)$ at $q$ (see \cite[Theorem 3.3]{FengHuang10}). This concludes the proof.
\end{proof}

\begin{remark}
The function $M(\bA,\cdot)$ is differentiable whenever there is just one equilibrium measure, and this is known to hold under several conditions, such as irreducibility or strict positivity, see \cite{FengLau02, FengKaenmaki11}.
\end{remark}

Using \cite[Theorem 4.8]{FengHuang10}, which is a higher dimensional version of \eqref{eq:MF-Lyapunov}, it is possible to obtain the following generalization of Proposition \ref{prop:continuity-Lyap-spec} to the joint spectrum of all Lyapunov exponents. Given $\bA\in\cG_{d,k}$ and $\mathbf{a}=(a_1,\ldots,a_d)\in\R^d$, let
\[
\Delta_{\bA}(\mathbf{a}) = \left\{ x\in X: \lim_{n\to\infty} \frac1n \log(\alpha_i(A(x,n)) = a_i-a_{i-1} \text{ for } i=1,\ldots, d \right\},
\]
where we define $a_0=0$.
\begin{proposition}
 Let $\bA_\ell$, $\bA\in \cG_{d,k}$ with $\lim_{\ell\to \infty}\bA_\ell=\bA$.  Assume that
 ${\mathbf a}:=\nabla P(\bA,\mathbf{t})$ exists at  some ${\mathbf t} \in \R^{d}_+$. Then for any ${\bf t}_\ell\to {\bf t}$,
 $$
 \lim_{\ell\to \infty}h_{\rm top}\left( \Delta_{\bA_\ell}({\bf a}_\ell)\right)=h_{\rm top}( \Delta_{\bA}(\bf a)),
 $$
 where ${\bf a}_\ell$ is  any extreme point in the set of the subdifferential of $M(\bA^{(\ell)},\cdot)$ at ${\bf t}_\ell$.
\end{proposition}
The proof is analogous to that of Proposition \ref{prop:continuity-Lyap-spec} and is omitted.

\subsection{Further applications and remarks}
\label{subsec:further-applications}

Theorem \ref{Falconer-thm} has been generalized in many different directions. Several of these generalizations involve a notion of pressure which fits into the framework of Section~\ref{subsec:generalpressure}; we make a short summary.

 In \cite[Theorem 5.3]{Falconer94},  Falconer showed that the dimension of  a mixing repeller of a non-conformal $C^2$ mapping $g$, under a certain distortion condition, is bounded above by the zero point of the pressure function $P(A_g,s)$, where $A_g$ is a matrix cocycle (defined on a subshift of finite type) generated by the derivative of $g$. He also showed that the box counting dimension of the mixing repeller equals the zero of this pressure under certain additional assumptions. As proved by Zhang \cite{Zhang97},  the upper bound remains  valid for  any $C^1$ mapping without additional assumptions.  Theorem \ref{thm:general-result} shows that functions $g$ which are piecewise affine on the attractor are continuity points of $g\to P(A_g,s)$. We do not know what happens for general $C^1$ maps.

 Later, K{\"a}enm{\"a}ki and Vilppolainen obtained a dimension formula  for  typical sub-self-affine sets \cite{KaenmakiVilppolainen10} via the zero point of $P(A, s)$, where $A$ is a locally constant matrix cocycle defined on a subshift space. When this subshift is of finite type our results are again applicable.

 Similar dimensional results (but in a random version) have been obtained for  random subsets of typical self-affine sets \cite{FalconerMiao10}. Here again there is a natural pressure which fits into the framework of Theorem \ref{thm:general-result}, with $g$ a locally constant function related to the random probabilities in the model. More recently, a dimension formula was obtained for random  affine code tree fractals \cite{JJKKSS12}.

 In \cite{Falconer99},  Falconer also introduced a family of pressure functions $P_q(\mathbf{A},s)$ with $q\in\R$ and showed their relation to the $L^q$ spectrum of self-affine measures. In particular, for $1<q<2$, Falconer showed that the $L^q$ dimension of typical self-affine measures equals the zero of this pressure. Later on \cite{Falconer10}, Falconer proved that this formula remains valid  for almost self-affine measures (which are a kind of random perturbation of self-affine measures) on the full range $q>1$.  For $q\in (0,1)$, the pressure function $P_q(\mathbf{A},s)$ fits into the framework of Theorem \ref{thm:generalization}, and we deduce joint continuity in $(\mathbf{A},s,q)$. Unfortunately we do not know if continuity of the pressure holds for $q>1$ which is the more interesting range in light of Falconer's results.

\bigskip

\textbf{Acknowledgement}. We thank the referee for helpful suggestions and remarks.


\end{document}